\documentclass[10pt]{article}
\usepackage{amsmath}
\usepackage[T1]{fontenc}
\usepackage[utf8]{inputenc}
\usepackage{caption}
\usepackage{subcaption}
\usepackage{amsfonts}
\usepackage{amssymb}
\usepackage{xcolor}
\usepackage{float}
\usepackage{bbm}
\usepackage{wrapfig}
\usepackage[english]{babel}
\usepackage{graphicx}
\usepackage{amsthm}
\usepackage{hyperref}
\usepackage{titling}
\usepackage{accents}
\usepackage [autostyle, english = american]{csquotes}
\MakeOuterQuote{"}
\numberwithin{equation}{section}
\theoremstyle{plain}
\newtheorem*{theorem*}{Theorem}
\newtheorem*{lemma*}{Lemma}
\newtheorem{theorem}{Theorem}

\newtheorem{lemma}{Lemma}[section]
\newtheorem{corollary}[lemma]{Corollary}

\newtheorem{conj}[lemma]{Conjecture}

\newtheorem{proposition}[lemma]{Proposition}

\theoremstyle{definition}
\newtheorem{definition}[lemma]{Definition}
\newtheorem{remark}[lemma]{Remark}

\newcommand{\be}{\begin{equation}}
\newcommand{\ee}{\end{equation}}
\newcommand{\ba}{\begin{array}{l}}
\newcommand{\ea}{\end{array}}
\newcommand{\Rr}{{\mathbb R}}

\newcommand{\prt}[1]{\left( #1 \right)}
\newcommand{\crch}[1]{\left[ #1 \right]}
\newcommand{\norm}[1]{\left\Vert #1 \right\Vert}
\newcommand{\abs}[1]{\left| #1 \right|}

\renewcommand{\div}{{\mbox{div}\,}}

\newcommand{\sgn}{{\mbox{sgn}\,}}

\def\Xint#1{\mathchoice
{\XXint\displaystyle\textstyle{#1}}%
{\XXint\textstyle\scriptstyle{#1}}%
{\XXint\scriptstyle\scriptscriptstyle{#1}}%
{\XXint\scriptscriptstyle\scriptscriptstyle{#1}}%
\!\int}
\def\XXint#1#2#3{{\setbox0=\hbox{$#1{#2#3}{\int}$ }
\vcenter{\hbox{$#2#3$ }}\kern-.6\wd0}}

\def\dashint{\Xint-}

\definecolor{darkred}{rgb}{0.6,0.1,0.1}
\definecolor{darkgreen}{rgb}{0.1,0.6,0.1}
\definecolor{darkblue}{rgb}{0.1,0.1,0.6}
\numberwithin{equation}{section}

\usepackage{geometry}

\title{On the long-time behavior of scale-invariant solutions to the 2d Euler equation and applications}
\author{Tarek M. Elgindi, Ryan W. Murray, Ayman R. Said}

\begin{document}

\maketitle

\begin{abstract}
    We study the long-time behavior of scale-invariant solutions of the 2d Euler equation satisfying a discrete symmetry. We show that all scale-invariant solutions with bounded variation on $\mathbb{S}^1$ relax to states that are piece-wise constant with \emph{finitely} many jumps.
     All continuous scale-invariant solutions become singular and homogenize in infinite time. On $\mathbb{R}^2$, this corresponds to generic infinite-time spiral and cusp formation.  The main tool in our analysis is the discovery of a monotone quantity that measures the number of particles that are moving away from the origin.
    This monotonicity also applies locally to solutions of the 2d Euler equation that are $m$-fold symmetric ($m\geq 4$) and have radial limits at the point of symmetry. 
     
     Our results are also applicable to the Euler equation on a large class  of surfaces of revolution (like $\mathbb{S}^2$ and $\mathbb{T}^2$). Our analysis then gives \emph{generic} spiraling of trajectories and infinite-time loss of regularity for globally smooth solutions on any such smooth surface, under a discrete symmetry.
\end{abstract}

\section{Introduction}
We are concerned with 2d inviscid flows; namely, solutions to the incompressible Euler equation:
\begin{equation}
    \label{2dEuler1} \partial_t \omega + u\cdot\nabla \omega=0,
\end{equation}
\begin{equation}\label{2dEuler2}
    u=\nabla^\perp \Delta^{-1}\omega.
\end{equation}
Here, the scalar vorticity $\omega:\mathbb{R}^2\times \mathbb{R}\rightarrow\mathbb{R}$ is transported by the velocity field $u:\mathbb{R}^2\times\mathbb{R}\rightarrow\mathbb{R}^2$ which is uniquely determined at each time $t\in\mathbb{R}$ from $\omega$ using the Newtonian potential:
\begin{equation}\label{Newtonian}u(x,t)=\frac{1}{2\pi}\int_{\mathbb{R}^2}\frac{(x-y)^\perp}{|x-y|^2}\omega(y,t)dy.\end{equation} We adopt the standard notation $v^\perp=(-v_2, v_1)$ for $v=(v_1,v_2)\in\mathbb{R}^2.$ It is well known that smooth enough solutions to the 2d Euler equation \eqref{2dEuler1}-\eqref{2dEuler2} retain their smoothness for all finite times. Much less is known in the infinite-time limit. Except in very special (but important) cases, very little is known about the long-time behavior of solutions in the large. In fact, since the Euler equation is fundamentally a (non-linear and non-local) transport equation, there is a strong possibility that despite the plethora of possible initial states, most solutions "relax" in infinite time to simpler states. This has been established in perturbative regimes in the ground breaking work of Bedrossian and Masmoudi \cite{BedrossianMasmoudi} and later extensions by Ionescu and Jia \cite{IonescuJia1,IonescuJia2} and Masmoudi and Zhao \cite{MasmoudiZhao}. As for the generic long-time behavior of solutions, there are two natural conjectures (see \cite{Sverakcours} and \cite{Shnirelman} respectively and also the review articles \cite{ED-review,Shnirelman-turvey}) regarding the long time behavior of solutions to the 2d Euler equation, which we state on $\mathbb{T}^2$ for convenience:
\begin{conj}
As $t\rightarrow\pm\infty,$ generic solutions experience loss of compactness.
\end{conj}

\begin{conj}
The (weak) limit set of generic solutions consists only of solutions lying on compact orbits. 
\end{conj}

\noindent These two conjectures together state that most solutions should, on the one hand, "relax" in infinite time in that they should lose $L^2$ mass. On the other hand, these limiting states are conjectured to have compact orbits; i.e. they must be very special, such as steady states, time-periodic solutions, etc. While there appear to be no results in the literature proving either of these phenomena in large data settings, there are a few results on generic small scale creation. Of note is the result of Koch \cite{Koch} in which strong growth of H\"older and Sobolev norms of the vorticity is established near any background solution (stationary or time-dependent) for which the gradient of the flow map is unbounded in time. Yudovich also established (boundary induced) growth results under some mild assumption on the data near the boundary of the domain \cite{Yudovich} (see also \cite{MSY} for an extension of \cite{Yudovich}). There are also numerous important results on growth of solutions in the neighborhood of stable steady states \cite{Nadirashvili,Denisov, KiselevSverak, Zlatos, ED-review}.

The purpose of this work is to demonstrate one setting where generic relaxation and growth can be established rigorously and in full generality and, in particular, away from equilibrium. We do this in the setting of scale-invariant solutions, introduced in \cite{EJ-symmetries}, which are only assumed to have bounded vorticity. We find that, buried in the dynamics of these Euler solutions, there is a powerful relaxation mechanism that induces both an arrow of time and a major contraction of phase space.  Our main Theorems can be stated informally as
\begin{theorem}
Consider the 2d Euler equation on $\mathbb{R}^2.$
\begin{itemize}
\item The set of $C^1$ and $m$-fold symmetric initial data whose corresponding solution is unbounded in $C^1$ is dense, within symmetry. Moreover, particle trajectories generically form spirals in infinite time.
\item Scale-invariant solutions relax in infinite time, by virtue of a monotone quantity, to states with finitely many jump discontinuities. 
\end{itemize}
\end{theorem}
\noindent The first statement also holds on $\mathbb{S}^2$ and requires $m\geq 3$, while the second requires $m\geq 4$. To the best of our knowledge, these are the only available large data results on generic solutions to the 2d Euler equation (even within symmetry). 

To properly frame the discussion, we begin by recalling the definition and properties of scale-invariant solutions.
\subsection{Scale-Invariant Solutions}
Solutions to the 2d Euler equation enjoy a two-parameter family of scaling symmetries. Indeed, once $\omega$ solves \eqref{2dEuler1}-\eqref{2dEuler2}, we have that $\omega_{\lambda,\mu}$ defined by 
\[\omega_{\lambda,\mu}(t,x)=\frac{1}{\mu}\omega(\mu t,\lambda x)\] also solves the Euler equation whenever $\mu,\lambda\in (0,\infty).$ It is natural to consider solutions that are scale-invariant:
\begin{definition}
$\omega:\mathbb{R}^2\times\mathbb{R}\rightarrow\mathbb{R}$ is said to be \emph{scale-invariant} if $\omega(\cdot,\lambda x)=\omega(\cdot,x)$ for all $\lambda\in (0,\infty)$ and $x\in\mathbb{R}^2.$
\end{definition}
An unfortunate fact about non-trivial scale-invariant solutions is that they cannot decay at spatial infinity. This makes it challenging to make sense of \eqref{Newtonian}. A key observation from \cite{EJ-symmetries} was that we could give a rigorous meaning to \eqref{Newtonian} when the vorticity $\omega$ satisfies a discrete symmetry. Indeed, in the most extreme case, when $\omega$ is radially symmetric, the formula \eqref{Newtonian} becomes completely local and \emph{does not depend in any way} on the behavior of $\omega$ as $|x|\rightarrow\infty.$ This is also true in an asymptotic sense when the vorticity satisfies a discrete rotational symmetry and it is a fortunate fact that discrete rotational symmetries are propagated by the Euler equation.  Indeed, if $\omega(x,t)$ solves the 2d Euler equation then whenever $\mathcal{O}\in SO(2)$, $\omega_{\mathcal{O}}$ defined by 
\[\omega_{\mathcal{O}}(t,x)=\omega(t,\mathcal{O}x),\] is also a solution. We now give the precise definition of discrete symmetry:
\begin{definition} For $m\in\mathbb{N}$ a function
$\omega:\mathbb{R}^2\rightarrow\mathbb{R}$ is said to be $m$-fold symmetric if $\omega(\mathcal{O}_{m}x)=\omega(x),$ for all $x\in\mathbb{R}^2,$ where $\mathcal{O}_m\in SO(2)$ is the matrix corresponding to a counterclockwise rotation by angle $\frac{2\pi}{m}.$
\end{definition}
\noindent If $X$ is a space of functions on $\mathbb{R}^2$, we will denote by $X_m$ the space of functions in $X$ that are $m$-fold symmetric. 
We will now state the main theorem of \cite{EJ-symmetries}.
\begin{theorem*}[Main theorem of \cite{EJ-symmetries}]\label{EJSITheorem}
Assume that $\omega_0\in L^\infty_m(\mathbb{R}^2)$ for some $m\geq 3.$ Then, there is a unique weak solution $\omega\in C_{w_*}(\mathbb{R}; L^\infty_m(\mathbb{R}^2))$ to \eqref{2dEuler1}-\eqref{2dEuler2} with $\omega|_{t=0}=\omega_0.$
\end{theorem*}

Here the notation $C_{w_*}(\mathbb{R}; L_m^\infty(\mathbb{R}^2))$ denotes that the solution is continuous in time with values in $L^\infty_m$ with the weak-star topology. A direct corollary of the existence and uniqueness in Theorem \ref{EJSITheorem} is the existence and uniqueness of scale-invariant solutions from $m$-fold symmetric scale-invariant data. 
\begin{corollary}
If $\omega_0\in L^\infty_m$ is scale-invariant, then the corresponding unique solution to the 2d Euler equation is scale-invariant for all $t\in\mathbb{R}.$
\end{corollary}

\noindent Moreover, for a scale-invariant solution $\omega$ if we let $$\omega(t,r,\theta)=g(t,\theta),$$ then $g$ solves a relatively simple equation on $\mathbb{S}^1:$
\begin{equation}\label{SIEuler}\partial_t g + 2G \partial_\theta g =0,\end{equation}
\begin{equation}\label{SIBSLaw}(4+\partial_{\theta\theta})G =g,\end{equation}
where \eqref{SIBSLaw} is uniquely solvable due to the condition $m\geq 3.$

\subsection{Summary of results}
In this section we offer a summary of our results. These results are all consequences of the more complete characterization of possible asymptotic behavior given in Theorem \ref{0-hom Euler Cauchy thm}, but we state them here as separate results of independent interest.

Our first main result in the paper reveals the driving relaxation mechanism, which is due to a type of set-valued monotonicity when tracked in Lagrangian coordinates. 

\begin{theorem}\label{thm:expanding-sets}[Expansion and Contraction]
Let $g_0\in L^\infty_m(\mathbb{S}^1),$ with $m \geq 4$, and $g,G$ be a solution of \eqref{SIEuler},\eqref{SIBSLaw}. Denote by $\chi$ the Lagrangian flow-map associated to the advection equation \eqref{SIEuler} on $\mathbb{S}^1$, namely the solution to
\[\frac{d}{dt}\chi(t,\theta)=2G(t,\chi(t,\theta)),\]
\[\chi(\theta,0)=\theta.\]
Then, for all $\theta\in \mathbb{S}^1,$ there exists $T(\theta)\in [0,\infty]$ so that $\partial_\theta\chi(t,\theta)$ is increasing on $[0,T(\theta))$ and decreasing on $[T(\theta),\infty).$ In other words, the set
\[
C(t) := \{\theta :  \partial_t\partial_\theta \chi(t,\theta) < 0\}
\]
is non-decreasing in the sense of set inclusions.

\medskip

\noindent We call the set \[E=\{\theta: T(\theta)=+\infty\}
\] the \textbf{expanding set}, which is shown to be closed. The set \[C:=\{\theta: T(\theta)<\infty\}\] is called the (eventually) \textbf{contracting set}. It is open, and $\partial_\theta\chi(t,\theta)\searrow 0$ for $t\in [T(\theta), \infty)$ on $C$.

\noindent Furthermore either
\[
g \underset{t\rightarrow+\infty}{\rightharpoonup} \dashint_{\mathbb{S}^1}g_0  \ \text{ \underline{or} } \ E \text{ is finite}.
\]
\end{theorem}

We single out in the following corollary which demonstrates one way to summarize the set monotonicity using a monotone function.
\begin{corollary}\label{entropy 0-hom Euler} 
Let $g_0\in L_m^\infty(\mathbb{S}^1),$ with $m \geq 4$. If $g_0$ is not identically constant, then the function  
\[
S(t)=\left|\{\theta :  \partial_t\partial_\theta \chi(t,\theta) \leq 0\}\right|_{\theta}
\]
is strictly increasing and approaches $2\pi$, where $\left|\cdot \right|_{\theta}$ is the Lebesgue measure in $\theta$.
\end{corollary}
\begin{remark}
It is important to emphasize that the montonicity of $S$ above and of its two-dimensional analogue (Corollary \ref{Full Euler Entropy}) \emph{does not} contradict the formal time-reversibility of the Euler equation. Indeed, reversing time requires reversing the sign of the vorticity. This means that quantities that depend on the sign of some quantity, as in the definition of $S$, could be monotone increasing or decreasing under the Euler evolution. $S$ can be viewed as a monotone quantity near a point of high symmetry in 2d. It would be interesting to investigate whether there are global monotone quantities outside of symmetry.   
\end{remark}

The following figure illustrates the evolution of the eventually contracting set $C(t)$ associated to the evolution in Figure \ref{fig:homoclinic}.
\begin{figure}[H]\label{entropy homoclinic orbit}
    \centering
    \includegraphics[scale=0.4]{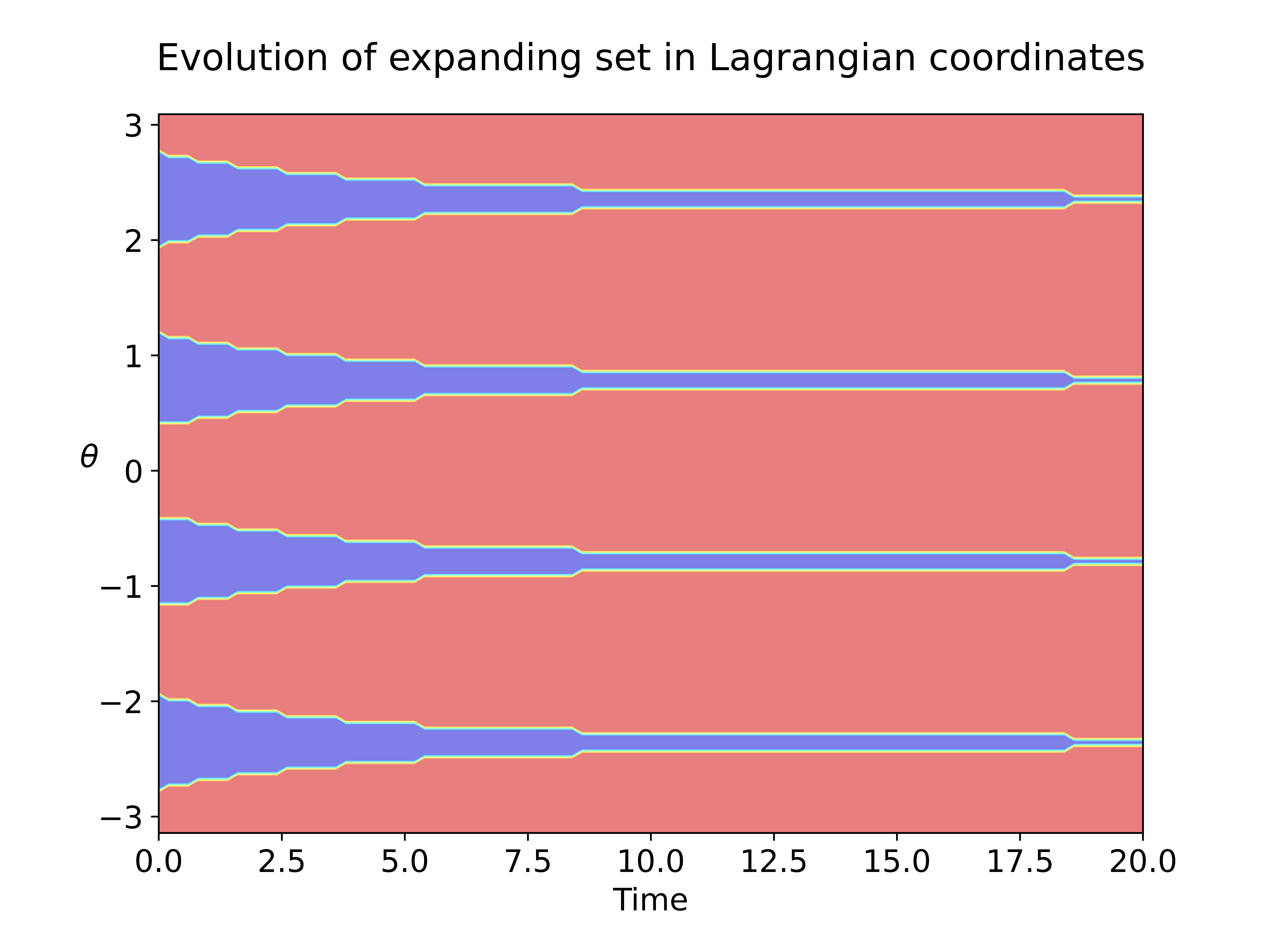}
    \caption{The expanding set associated to the evolution in Figure \ref{fig:homoclinic}}
    \label{fig:entropy homoclinic}
\end{figure}

Finally the increasing functional exhibited in Corollary \ref{entropy 0-hom Euler} generalizes to a large class of bounded  $m$-fold symmetric solutions of 2d Euler. First we recall the following slightly generalised proposition from \cite{EJ-symmetries} whose proof follows in verbatim from Proposition \ref{prop: movement of poles}.
\begin{proposition}\label{Full Euler link}
Consider $\omega_0 \in L_m^{\infty}(\mathbb{R}^2)$,  $m\geq 3$ and the associated unique global solution $\omega \in C_w\prt{\Rr,L_m^{\infty}(\mathbb{R}^2)}$ of \eqref{2dEuler1},\eqref{2dEuler2}, written in polar coordinates. Suppose that there exists $ g_0\in W^{1,\infty}_m(\mathbb{S}^1)$ such that for almost every $\theta\in[0,2\pi)$
\[
\lim_{r\rightarrow 0}\omega_0(r,\theta)=g_0(\theta).
\]
Then, there exists $g\in C_w\prt{\Rr,W^{1,\infty}_m(\mathbb{S}^1)}$ such that for all $t$ \[
\omega(t,r,\theta)\underset{r\rightarrow 0}{\longrightarrow} g(t,\theta) \text{ almost everywhere}.
\]
Moreover, $g$ solves \eqref{SIEuler} and \eqref{SIBSLaw} with initial data $g_0$.
\end{proposition}
\begin{corollary}\label{Full Euler Entropy}
Consider $\omega_0$ as in Proposition \ref{Full Euler link}. Set $u$ the divergence free velocity field associated to $\omega$ and $\phi$ the flow defined by 
\[
\frac{d}{dt}\phi=u(t,\phi),\ \phi(0,\cdot)=Id.
\]
Then if $g_0$ is not identically constant the following functional is a strictly increasing function  of time that goes from to $0$ to $2\pi$
\[
S(t)=\abs{\left\{\theta, \frac{d}{dt} \abs{\phi}(0,\theta) \geq 0 \right\}}_\theta.
\]
\end{corollary}

We now give the main relaxation theorem.
\begin{theorem}\label{thm:conv-trav-wave}
Let $g_0\in L^\infty_m(\mathbb{S}^1)$, with $m \geq 4$. Then either $g$ weakly converges to $\dashint_{\mathbb{S}^1}g_0$ in $L^p$ for all $p<+\infty$ or $E$ is finite. In the case that $E$ is finite, if moreover on all points in $E$ $g_0$ has left and right limits and those limits are distinct, then there exists an asymptotic profile $g_{+ \infty}(t,\cdot)$ which is piecewise constant with exactly $2|E|$ jumps solving \eqref{0-hom Euler Cauchy thm} such that 
\[
\abs{g(t,\theta)-g_\infty(t,\theta)}\underset{t\rightarrow+\infty}{\longrightarrow}0, \text{ for almost every }\theta.
\]
Finally if at least one point of $E$ is a continuity point of $g_0$ then the asymptotic profile $g_{+ \infty}$ is identically constant and equal to $\dashint_{\mathbb{S}^1} g_0 ,$ and the convergence occurs point-wise almost everywhere. 
\end{theorem}
\begin{remark}
The previous relaxation theorem thus contains the case of bounded functions that admit left and right limits at each point which is precisely the Banach space of regulated functions, $Reg$, realized as the closure of $BV$ for the $L^\infty$ norm. We have trivially $C^0(\mathbb{S}^1)\subset Reg(\mathbb{S}^1)$. The convergence on $Reg$ happens point-wise almost everywhere. We also remark that, while Proposition \ref{Full Euler link} requires $g_0\in W^{1,\infty},$ the result can be extended to the case where $g_0$ is piecewise $C^{1,\alpha}$ as was done in \cite{EJSVPI}. 
\end{remark}

{\bf Homoclinic and Heteroclinic Orbits:} Our results also reveal the existence of homoclinic and heteroclinic orbits. In Figure \ref{fig:homoclinic}, we give a plot of the homoclinic orbit connecting a constant state to itself in the limits $t\rightarrow\pm\infty.$ All continuous solutions lie on such homoclinic orbits by Theorem \ref{thm:conv-trav-wave}.
\begin{figure}[H]\label{fig:homoclinic}
    \centering
    \includegraphics[scale=0.4]{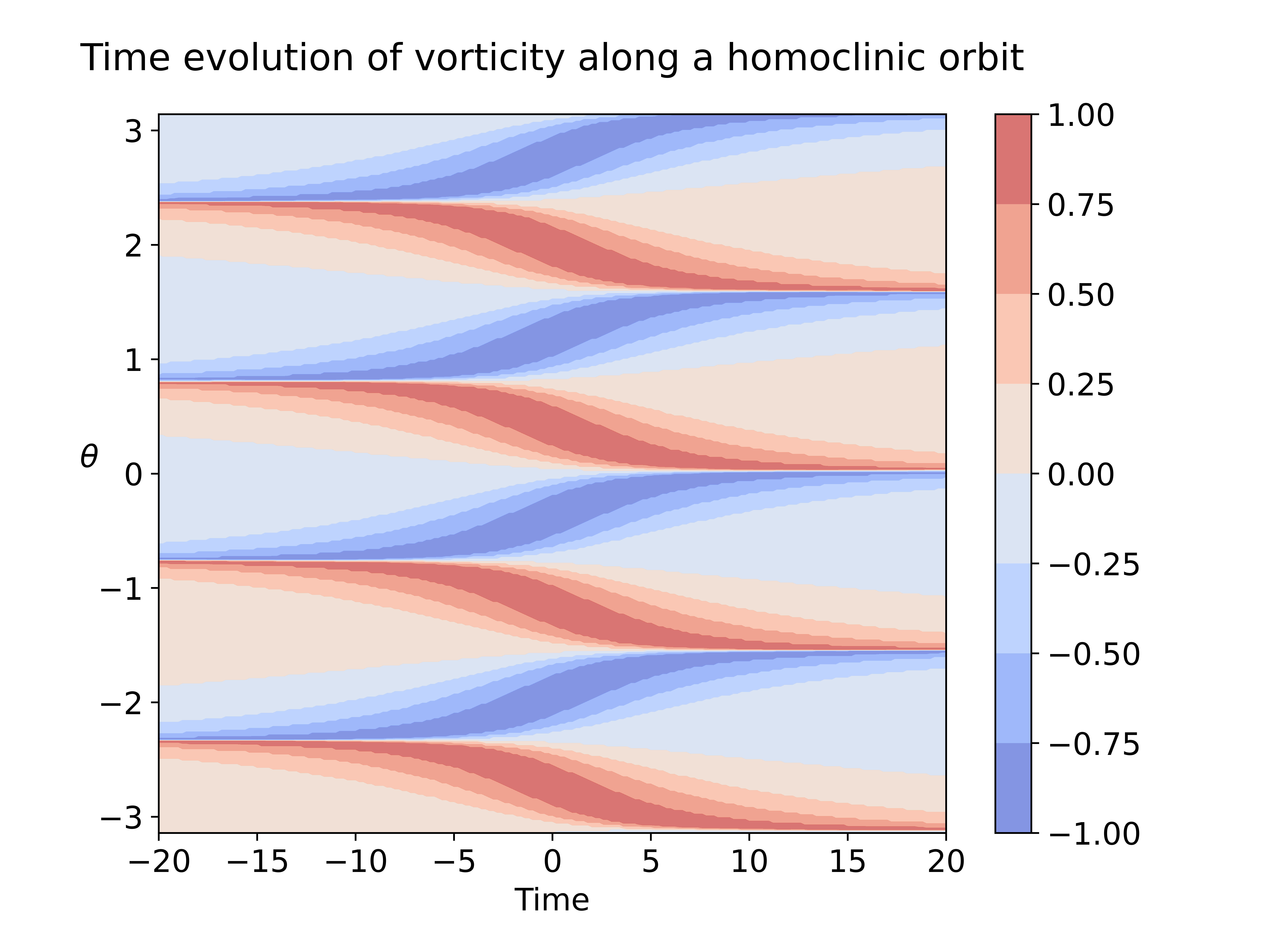}
    \caption{An orbit of the system which approaches zero at $t = \pm \infty$.} 
    \label{fig:homoclinic}
\end{figure}
For heteroclinic orbits, the end states $g_{\pm \infty}$ are completely determined by the forward and backward expanding sets $E_{\pm}$ and that
\[
E_{+\infty}\subset \{\theta,\partial_\theta G_0>0\}  \text{ and } E_{-\infty}\subset \{\theta,\partial_\theta G_0<0\}
\]
It is thus easy to construct examples showing that the map $g_0\mapsto g_{\pm \infty}$ is a non-trivial map with genuine quantifiable loss of information (cf Corollary \ref{entropy 0-hom Euler}). Figure \ref{fig:heteroclinic} illustrates an orbit of the system connecting two different states towards $t = \pm \infty$, one being a jump state with sixteen jumps and the other having eight jumps.

\begin{figure}[H]
    \centering
    \includegraphics[scale=0.4]{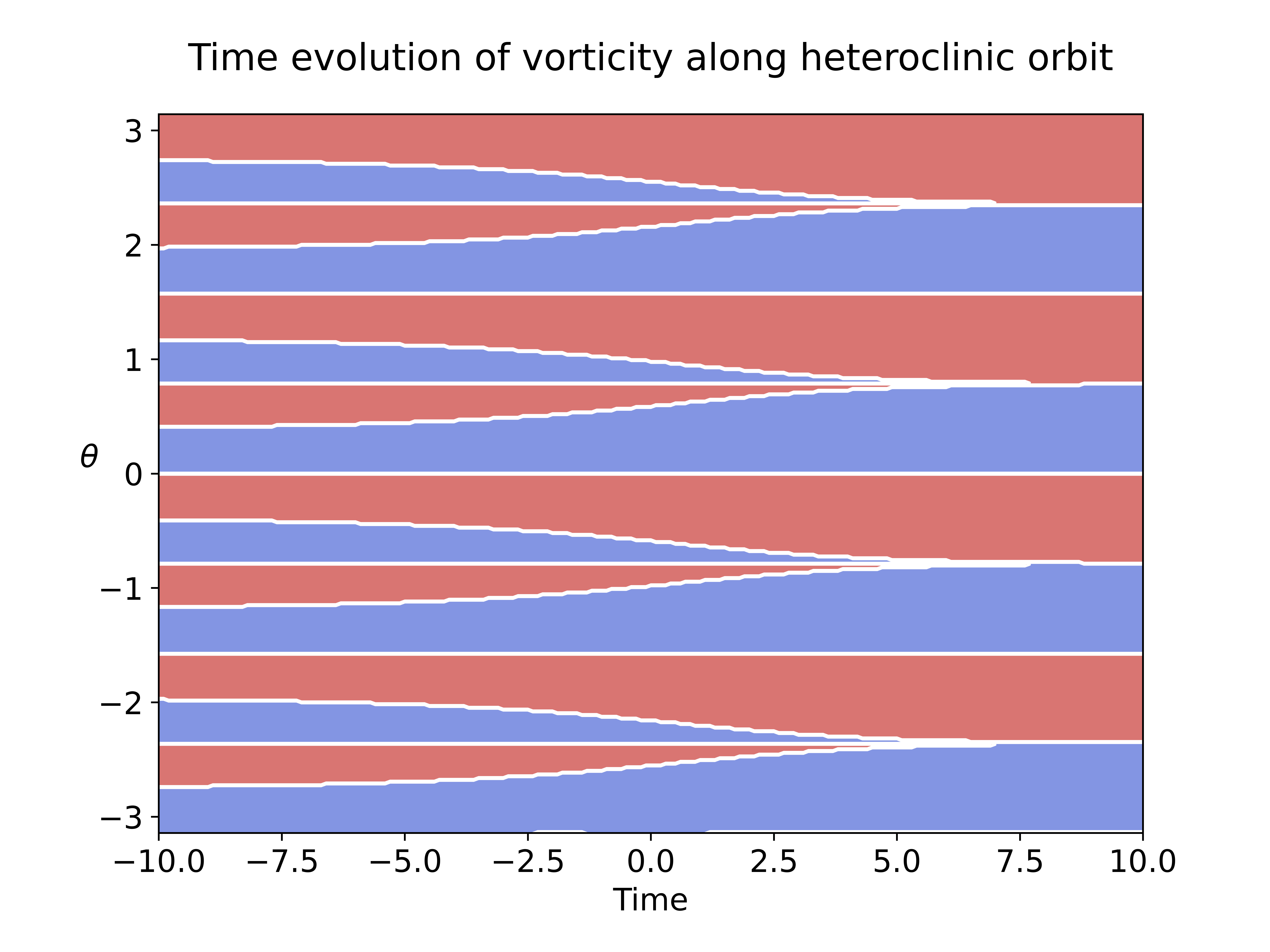}
    \caption{An orbit of the system connecting two different steady states towards $t = \pm \infty$.}
    \label{fig:heteroclinic}
\end{figure}

We also provide a complete complete classification of the possible limiting travelling wave profiles and steady states of the system \eqref{SIEuler}, \eqref{SIBSLaw}.
\begin{theorem}\label{thm:class-steady-states}
Let $g\in L^\infty_m(\mathbb{S}^1)$, with $m \geq 4$ be a steady solution to the steady state problem
\[
\begin{cases}
G\partial_\theta g=0\\
4G+\partial_{\theta \theta }G=g
\end{cases}.
\]
Then either $g=0$ or g is piecewise constant and has a finite number of jumps occurring at the global maxima and minima of $\partial_\theta G$ which have exactly opposite non zero values.
\end{theorem}
{\bf Generic infinite in time blow up for smooth solutions on $\mathbb{R}^2$ and $\mathbb{S}^2$:}
Using Proposition \ref{Full Euler link} to analyze the winding of trajectories around $0$ and an important result of Koch \cite{Koch}, we establish generic $C^1$ growth
\begin{theorem}\label{R2Theorem}
Consider the class $X=C^1_m\cap L^1(\mathbb{R}^2), \ m\geq 3.$ Then, the set of all initial data in $C^1_m$ whose Euler solution diverges in $C^1$ as $t\rightarrow\infty$ is of second category (in particular, such data are dense in $X$).  
\end{theorem}
The basic idea in the proof of Theorem \eqref{R2Theorem} is that trajectories around $0$ and around the point at infinity are generically winding with different speeds. The speed of local winding around $0$ is given precisely by $\omega_0(0)$ (because of the $m$-fold symmetry), while the winding at infinity is winding with speed zero due to the fact that the velocity is uniformly bounded. This leads to spiral formation and unbounded growth of the gradient of the flow-map. 

A corresponding result holds on $\mathbb{S}^2$ since generically the winding around the north and south poles are different. We prove in Proposition \ref{prop: movement of poles} that, under $m$-fold symmetry symmetry, the trajectories of particles around the north and south poles are given asymptotically by two scale invariant Euler flows, which are \emph{generically} different. A consequence is generic unbounded growth in $C^1$ for $m$-fold symmetric solutions on $\mathbb{S}^2.$
\begin{theorem}\label{S2Theorem}
Consider the class $C^1_m(\mathbb{S}^2),\ m\geq 3$ of $C^1$ and $m$-fold symmetric vorticities. Then, the set of all initial data in $C^1_m$ whose Euler solution diverges in $C^1$ as $t\rightarrow\infty$ is of second category (in particular, such data are dense in $C^1_m$).  
\end{theorem}

\begin{remark}
\begin{itemize}
Two remarks are in order. 
    \item We emphasize again that there is no smallness hypothesis on the initial data.
    \item The analyses made here generalizes to compact smooth surfaces of revolution that posses at least two points at the axis of symmetry, as around those points the surface is necessarily flat and it can be analogously shown that the angular trajectories of particles around those points are given by two distinct scale invariant Euler flows.
\end{itemize}

\end{remark}
The following Figure \ref{fig:evo on the sphere} illustrates a caricature of the spiral formation associated to the growth exhibited in Theorem \ref{S2Theorem}.
\begin{figure}[H]
     \centering
     \begin{subfigure}[b]{0.4\textwidth}
         \centering
         \includegraphics[width=\textwidth]{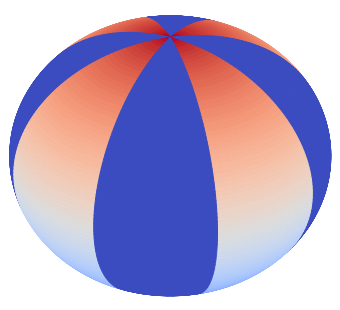}
     \end{subfigure}
     \hspace{.2in}
     \begin{subfigure}[b]{0.4\textwidth}
         \centering
         \includegraphics[width=\textwidth]{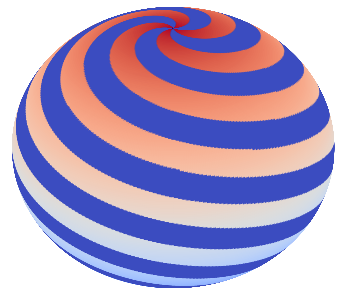}
     \end{subfigure}
        \caption{Evolution on the Sphere. On the left an initial vorticity is displayed, and on the right a caricature of how the vorticity will look at large time.}
        \label{fig:evo on the sphere}
\end{figure}



\section{The zero-homogeneous Euler System: the theorems}

We now state, in full generality, the theorem describing the asymptotic behavior of solutions of the zero-homogeneous Euler equation.

\begin{theorem}\label{0-hom Euler Cauchy thm}
Consider $g_0\in L^\infty_m(\mathbb{S}^1)$ with $m\geq 4$. Then the Cauchy problem
\begin{equation}\label{0 homogeneous limit system with initial data}
\begin{cases}
\partial_t g+2G\partial_{\theta}g=0\\
4G+\partial_{\theta \theta}G=g,
\end{cases}
\text{ with } g(0,\cdot)=g_0,
\end{equation}
admits a unique global in time solution $g\in L^\infty(\mathbb{R}, L^\infty_m(\mathbb{S}^1))\cap C^0(\mathbb{R}, L_m^1(\mathbb{S}^1))$. Furthermore, the limiting behavior of $g$ when $t$ goes to $+\infty$ is given as follows. First, defining the characteristic flow 
\[
\frac{d}{dt}\chi(t,\theta)=2G\prt{\chi(t,\theta)}, \ \chi(0,\theta)=0.
\]
We decompose
\[
\mathbb{S}^1=E\cup C, \text{ with } C=\left\{\theta,\ \partial_\theta\chi(t,\theta)\underset{t\rightarrow+\infty}{\longrightarrow0}\right\}
\]
and let
\[
E=\{\theta, \text{ for all }t,\ \partial_\theta\chi(t,\theta) \text{ is strictly increasing in } t\}.
\]
Moreover, either
\[
g \underset{t\rightarrow+\infty}{\rightharpoonup} \dashint g_0  \ \text{ \underline{or} } \ E \text{ is finite}.
\]
In the case where $E$ is finite, with $E=\{a_{0},a_2,a_4,\cdots,a_{2(n-1)}\}$, then 
\begin{itemize}
    \item If $g_0$ is continuous at any point of $E,$ then $g$ relaxes point-wise almost everywhere to $\dashint_{\mathbb{S}^1} g_0$.
    \item If on all points of $E$, $g_0$ has one-sided limits with distinct left and right limits, then $g$ relaxes point-wise almost everywhere to a piecewise constant profile with $2\abs{E}$ jumps where the value at the jumps are exactly given by $\displaystyle \prt{\lim_{\epsilon \rightarrow 0}g_0({a_{2i}-\epsilon}),\lim_{\epsilon \rightarrow 0}g_0(a_{2i}+\epsilon)}_{0\leq i\leq n-1}$. 
\end{itemize}

\end{theorem}

The next theorem gives a complete classification of steady states of the zero-homogeneous Euler equation in $L^\infty_m(\mathbb{S}^1)$.
\begin{theorem}\label{Steady states of the zero homogeneous Euler equation}
The set of solutions $g\in L^\infty_m(\mathbb{S}^1)$, $m \geq 4$ of the 
steady equation
\[
G\partial_\theta g=0 \text{ with } 4G+\partial_{\theta \theta}G=g.
\]
is given by the set of piecewise constant functions which have a finite even number of jumps occurring at the global maxima and minima of $\partial_\theta G$, which have exactly opposite non-zero values.

Concretely, in the case where $g$ has a finite number of jumps at the points $a_0=-\frac{\pi}{m}<a_1<\cdots<a_{n-1}<a_{2n}=\frac{\pi}{m}$, we may write
\[
g(\theta)=\sum^{2n}_{i=1}g_i\mathbbm{1}_{(a_{i-1},a_i)}(\theta), \text{ with }g_i \in \mathbb{R} \text{ and }g_i\neq g_{i+1}.
\]
Then the requirement upon the jumps occurring at the global extrema of $\partial_\theta G$ corresponds to the requirement that, for all $i\in \{1,\cdots,2n\}$,
\begin{equation}\label{System for Steady states}
\sum^{2n}_{j=1}g_j\abs{\cos\left(\frac{m}{2}(a_i-a_{j})\right)-\cos\left(\frac{m}{2}(a_i-a_{j-1})\right)}
=\frac{m(16-m^2)}{24\pi}\sum^{2n}_{j=1}g_j(a_j-a_{j-1}).
\end{equation}
Moreover an ordered sequence $(g^i)_{1\leq i\leq 2n}$  such that $g_i  \in \mathbb{R^*}, \ g_i\neq g_{i+1},$ determines through the system \eqref{System for Steady states} a unique increasing sequence $(a_i)_{0\leq i \leq 2n}\in \crch{-\frac{\pi}{m},\frac{\pi}{m}}$ and $c\in \Rr$ such that $$ \displaystyle g(t,\theta)=\sum^n_{i=1}g_i\mathbbm{1}_{(a_{i-1},a_i)}(\theta-ct)$$ is a m-fold symmetric solidly rotating steady solution to \eqref{0-hom Euler Cauchy thm}.
\end{theorem}

\begin{remark}

It is important to highlight that all of the stated results in this section are for $m\geq 4$, despite the fact that solutions are well-defined for $m\geq 3.$ It can be shown that all of the statements above still hold in the case $3-$fold odd symmetric data. It can also be shown that, as stated, all of the previous results fail for general $3-$fold symmetric data. The dynamics of solutions when $m=3$ may be more complicated and this is an interesting case for further investigation.  

\end{remark}

\subsection{Proof of Theorem \ref{0-hom Euler Cauchy thm} }

We now turn our attention to proving our first main theorem. Before we begin, we recall that the well-posedness theory for these equations with data in $L_m^\infty(\mathbb{S}^1)$, with $m \geq 3$, can be found in \cite{EJ-symmetries}. 

\subsubsection{The Riccati structure }
To elucidate the underlying Riccati structure in the evolution, we compute the evolution equation on $$\partial_\theta G=\partial_\theta(4+\partial_{\theta \theta})^{-1}g.$$
\paragraph{\textbf{$\bullet$ The equation on $\partial_\theta G$.}}
We rewrite the system on $g$ as follows
\[
\partial_t (\partial_{\theta \theta}+4)G+2G\partial_\theta (\partial_{\theta \theta}G+4G)=0\Rightarrow \partial_t (\partial_{\theta \theta}+4)G+4\partial_\theta G^2+2G \partial_{\theta \theta \theta}G=0.
\]
We then compute 
\[
2G \partial_{\theta \theta \theta}G=2\partial_{\theta}(G\partial_{\theta \theta}G)-2\partial_{\theta}G\partial_{\theta \theta}G=\partial_{\theta}[2\partial_\theta(G\partial_{\theta }G)-2(\partial_{\theta }G)^2]-\partial_\theta(\partial_{\theta}G)^2=\partial_{\theta \theta \theta}G^2-3\partial_\theta(\partial_{\theta}G)^2.
\]
Substituting back into the equation we get
\[
\partial_t (\partial_{\theta \theta}+4)G+(\partial_{\theta \theta}+4)\partial_\theta G^2-3\partial_\theta (\partial_\theta G)^2=0,
\]
and thus,
\[
\partial_t G+\partial_\theta G^2-3\frac{\partial_\theta}{\partial_{\theta \theta}+4} (\partial_\theta G)^2=0.
\]
This finally gives
\[
\partial_t \partial_\theta G+2G\partial_\theta \partial_\theta G+2(\partial_\theta G)^2-3\frac{\partial_{\theta\theta}}{\partial_{\theta \theta}+4} (\partial_\theta G)^2=0.
\]
To isolate the leading order term we write
\[
\partial_t \partial_\theta G+2G\partial_\theta \partial_\theta G-(\partial_\theta G)^2+\frac{12}{\partial_{\theta \theta}+4} (\partial_\theta G)^2=0.
\]
We define
\[
c(t,\theta)=\frac{12}{\partial_{\theta \theta}+4} (\partial_\theta G)^2.
\]
 In order to exhibit the Lyapunov structure, it is crucial for us to have $c(t,\theta) >0$, which is established by the following lemma whenever $g$ is not identically constant. 
\begin{lemma}\label{elliptic estimate}
For non-negative $h\in C(\mathbb{R})$ that is $\frac{2\pi}{m}$ periodic with $m\geq 4$,
\[
c_h:=\frac{12}{\partial_{\theta \theta}+4}h\geq \tilde C_m \dashint_{-\pi}^\pi h,
\] where $\tilde C_m>0$ when $m\geq 5$ and $\tilde C_4=0.$
Moreover, if there exists $\theta$ such that $c_h(\theta)=0$ then necessarily $h=0$.
\end{lemma}
\begin{proof}
The kernel for $\frac{1}{\partial_{\theta\theta}+4}$ operating on functions that are orthogonal to $\sin(2\theta)$ and $\cos(2\theta)$ is given in \cite{EJ-symmetries}:
\begin{equation}\label{kernel formula}
\frac{1}{\partial_{\theta \theta}+4}h(\theta)=\frac{1}{2\pi}\int^{\pi}_{-\pi} K_{\frac{1}{\partial_{\theta \theta}+4}}(\theta-\omega)h(\omega)d\omega, \text{ where }
\end{equation}
\begin{equation}
    K_{\frac{1}{\partial_{\theta \theta}+4}}(\theta)=\frac{\pi}{2}\sin(2\theta)\frac{\theta}{|\theta|}-\frac{1}{2}\sin(2\theta)\theta-\frac{1}{8}\cos(2\theta).
\end{equation}
In the case of $m$-fold symmetry we may further write
\[
\frac{1}{\partial_{\theta \theta}+4}h(\theta)=\frac{m}{2\pi}\int^{\frac{\pi}{m}}_{-\frac{\pi}{m}} {K^m_{\frac{1}{\partial_{\theta \theta}+4}}}(\theta-\omega)h(\omega)d\omega,
\]
where 
\[
{K^m_{\frac{1}{\partial_{\theta \theta}+4}}}(\theta)=\frac{1}{m}\sum^{m-1}_{j=0}{K_{\frac{1}{\partial_{\theta \theta}+4}}}\left(\theta+j\frac{2\pi}{m}\right)=C_m\left|\sin\left(\frac{m}{2}\theta\right)\right|+\tilde{C}_m.
\]
The result will follow once we show that $C_m>0$ and $\tilde C_m\geq 0.$ In fact, we will show that $C_m>0$ for any $m\geq 3$ and $\tilde C_m>0$ for any $m\geq 5,$ while $\tilde C_4=0$ and $\tilde C_3<0.$
First to compute $C_m$, we note that
\[
\frac{1}{\partial_{\theta \theta}+4}\sin(m\theta)=\frac{1}{4-m^2}\sin(m\theta)=C_m\frac{m}{2\pi}\int^{\frac{\pi}{m}}_{-\frac{\pi}{m}} \left|\sin\left(\frac{m}{2}(\theta-\omega)\right)\right|\sin(m\omega)d\omega.
\]
Then 
\begin{align*}
&\int^{\frac{\pi}{m}}_{-\frac{\pi}{m}} \left|\sin\left(\frac{m}{2}(\theta-\omega)\right)\right|\sin(m\omega)d\omega=\int^{\frac{\pi}{m}}_{-\frac{\pi}{m}} \left|\sin\left(\frac{m}{2}\omega\right)\right|\sin(m(\theta-\omega)d\omega\\
&=\int^{\frac{\pi}{m}}_{-\frac{\pi}{m}} \left|\sin\left(\frac{m}{2}\omega\right)\right|\cos(m\omega)d\omega\sin(m\theta)-\underbrace{\int^{\frac{\pi}{m}}_{-\frac{\pi}{m}} \left|\sin\left(\frac{m}{2}\omega\right)\right|\sin(m\omega)d\omega\cos(m\theta)}_{=0}\\
&=\frac{2}{m}\int^{\frac{\pi}{2}}_{-\frac{\pi}{2}} \left|\sin\left(\omega\right)\right|\cos(2\omega)d\omega\sin(m\theta)=-\frac{4}{3m}\sin(m\theta).
\end{align*}
Thus we get for $C_m$
\[
\frac{1}{4-m^2}=-C_m\frac{4}{3m}\frac{m}{2\pi} \Longleftrightarrow C_m= \frac{3\pi}{2(m^2-4)}.
\]
We proceed to compute $\tilde{C}_m$ by noting
\[
\frac{1}{\partial_{\theta \theta}+4}1=\frac{1}{4}=C_m\frac{m}{2\pi}\int^{\frac{\pi}{m}}_{-\frac{\pi}{m}} \left|\sin\left(\frac{m}{2}(\omega)\right)\right|d\omega+\tilde{C}_m \longrightarrow \tilde{C}_m=\frac{1}{4}-\frac{2C_m}{\pi}.
\]
\end{proof}
\begin{remark}
The failure of the results for $m=3$ stems from the fact that $\tilde{C}_3<0$ and thus $K^3_{\frac{1}{\partial_{\theta\theta}+4}}$ is not signed.
\end{remark}
Henceforth we will suppose that $g_0$ is not identically constant, as otherwise all of the stated results are trivial. Returning to the equation on $\partial_\theta G$, we now have that
\[
\partial_t \partial_\theta G+2G\partial_\theta \partial_\theta G-(\partial_\theta G)^2+c=0,
\]
where $c$ is a strictly positive function from the previous lemma. We now define the flow map
\[
\frac{d}{dt}\chi=2G\circ \chi, \ \  \chi(0,\cdot)=Id.
\]
Then by defining $F=\partial_\theta G\circ \chi$ and differentiating the previous equation in $\theta$ we obtain the system 
\begin{equation}\label{eqn:riccati-euler}
\begin{cases}
\frac{d}{dt}F=F^2-\hat c\\
\frac{d}{dt}\partial_\theta \chi=2F \partial_\theta \chi
\end{cases} \text{ with } F(0,\theta)=\partial_\theta G(0,\theta)=\partial_\theta(4+\partial_{\theta \theta})^{-1} g_0(\theta),\ \hat c=c\circ \chi \text{ and } \partial_\theta \chi(0)=1.
\end{equation}
This Riccati equation essentially reduces our problem to solving a family of ODEs, parameterized by $\theta$. Furthermore, the signed structure of the forcing term $\hat c$ enables us to uncover critical monotone quantities in the equations in the following section.
\begin{remark}
An analogous Riccati structure has been identified in \cite{SarriaSaxton} for Proudman-Johnson equation where it was proved among other things that regular solutions converge to constants for the equation with an analogous scaling to the scale invariant 2d Euler equation. A key difference in technique is that the Riccati equation obtained in \cite{SarriaSaxton} does not depend on the analogous variable to $\theta$.
\end{remark}

\subsubsection{An Entropy: The expanding and contracting sets}
We now observe that, by the strict positivity of $\hat c$ along with the differential equation \eqref{eqn:riccati-euler}, that if $F(t,\theta)\leq 0$ for a given time $t$ then we will have that $F(\hat t, \theta) < 0$ for all $\hat t > t$. In other words, the following sets are increasing in the sense of inclusion:
\[
C_t=\{\theta, F(t,\theta)< 0\}.
\]
It is then natural to define $C=\cup_{t\in \Rr }C_t$ and $E=[-\pi,\pi]\setminus C$. One interpretation of $E$ and $C$, stemming from the evolution equation for $\partial_\theta \chi$, is that
\[C=\{\theta,\ \partial_\theta\chi(t,\theta) \text{ is eventually decreasing in }t\},
\]
and 
\[
E=\{\theta, \text{ for all }t,\ \partial_\theta\chi(t,\theta) \text{ is strictly increasing in } t\}.
\]
Thus we may interpret $E$ as the set of expanding points and $C$ the set of the (eventually) contracting points, all along the Lagrangian trajectories. We also note that as $F$ belongs to $W^{1,\infty}$ in both time and space that the function
\[
S(t)=|C_t|,
\]
is a strictly increasing function of  time.
\subsubsection{Detailed analysis of the expanding and contracting sets}
\paragraph{\textbf{Integrating the Riccati equation:}}
Following the classical approach to solving scalar Riccati equations, we now define $y$ such that \[F=-\partial_t\ln\left(y\right)=-\frac{\partial_ty}{y}\text{ thus }-\frac{\partial_ty(0,\theta)}{y(0,\theta)}=\partial_\theta(4+\partial_{\theta \theta})^{-1} g_0(\theta).\] At $t=0$ we have an extra degree of freedom which we fix by setting $y(0,\theta)=1$. Then, by using \eqref{eqn:riccati-euler} we have that 
\[
\partial_\theta \chi(t,\theta)=\frac{1}{y^2(t,\theta)}, \ \ g=g_0\circ \chi^{-1}(t,\theta),
\]
and furthermore we find that $y$ satisfies the second-order differential equation
\begin{equation}\label{eqn:y-ODE}
\partial_{tt}y=\hat cy, \ \ y(0,\theta)=1, \ \ \partial_t y(0,\theta)=-\partial_\theta(4+\partial_{\theta \theta})^{-1} g_0(\theta)=-\partial_\theta G_0(\theta).
\end{equation}
From the global well-posedness of the system \eqref{0 homogeneous limit system with initial data} we see that for finite $t$, $y(t,\theta)\neq0$, thus from $y(0,\theta)=1$ we have $y(t,\theta)>0$. As $\hat c$ is positive, we have that $y(\cdot,\theta)$ is strictly convex and is decreasing on interval $[0,T_0(\theta)]$ with $T_0(\theta)\in[0,+\infty]$. Note that if $T_0(\theta)<+\infty$ then $y(\cdot,\theta)$ necessarily converges to $+\infty$ and if not then $y(\cdot,\theta)$ converges to a $0\leq l(\theta)\leq 1$. Thus by definition we get 
\[C=\left\{\theta: \ T_0(\theta)<+\infty\right\}=\left\{\theta,\ \partial_\theta\chi(t,\theta)\underset{t\rightarrow+\infty}{\longrightarrow0}\right\},
\]
and
\[
E=\left\{\theta: \ T_0(\theta)=+\infty \right\} = \left\{y(t,\theta)\underset{t\rightarrow+\infty}{\longrightarrow}l(\theta) \right\}.
 \]
We now proceed to show that $E$ is non empty: we assume, for that sake of contradiction, that it is empty. Then we have that $T_0(\theta)<+\infty$ for all $\theta$. Furthermore, by the strict positivity of $\hat c$ we have that $T_0(\theta)$ is continuous, and hence $T_0(\theta)$ is bounded from above by a constant $\bar T$. Then for $t > \bar T$ we have that $\partial_\theta \chi(t,\theta)$ is monotone decreasing, and has limit zero. On the other hand, for every $t$ we have that
\begin{equation}\label{2 pi integral}
2\pi = \int_0^{2\pi} \partial_\theta\chi(t,\theta)d\theta.
\end{equation}
The monotone convergence theorem then immediately gives a contradiction, thus $E\neq \emptyset$.

Now to determine $l(\theta)$, exactly one of two scenarios are possible by Lemma \ref{ODELemma2} first observed in \cite{Sobol49}.
\begin{description}
\item[Scenario 1] Either $\int_0^{\infty}t\hat{c}(t,\theta)dt<+\infty$  and $l(\theta)>0$.
\item[Scenario 2] Or $\int_0^{\infty}t\hat{c}(t,\theta)dt=+\infty$ and $l(\theta)=0$.
\end{description}

\subsubsection{On the finiteness of the expanding set}\label{sub sec:cvg to steady state}

We now show that either the solution relaxes to the constant state or $E$ is finite. First we analyse the two possible scenarios.
\paragraph{\textbf{$\bullet$ Scenario 1.}}
\[
\exists \theta, \ \int_0^{+\infty}t\hat{c}(t,\theta)dt<+\infty,
\]
thus by Lemma \ref{elliptic estimate}
\[
\int_0^{+\infty}t\norm{\partial_\theta G}_{L^2}^2(t)dt<+\infty.
\]
Now noting that $\frac{d}{dt}\norm{\partial_\theta G}_{L^2}^2$ is uniformly bounded, $\partial_\theta G\underset{t\rightarrow +\infty}{\longrightarrow}0$ in $L^2$ thus $g$ converges to $\dashint_{\mathbb{S}^1}g_0$ in $H^{-1}$. Thus $\int_0^{+\infty}t\hat{c}(t,\theta)dt<+\infty$ is true for all $\theta$. In this case by Lemma \ref{ODELemma3} for all $\theta\in E$, $y(t,\theta)$ decreases to $l(\theta)>0$ and by \eqref{2 pi integral} we get that $|E|>0.$ The $L^p$ norms of $g$ being uniformly bounded imply that $g$ converges weakly to $\dashint_{\mathbb{S}^1}g_0$ in $L^p$ for all $1\leq p<+\infty$.

\paragraph{\textbf{$\bullet$ Scenario 2.}}
Henceforth we suppose 
\[
\forall \theta, \int_0^{+\infty}t\hat{c}(t,\theta)dt=+\infty,
\]
thus by Lemma \ref{ODELemma3} for all $\theta \in E$, $y(t,\theta)$ decreases to $0$. The ODE analysis implies that the decomposition $\mathbb{S}^1=E\cup C$ now reads:
 \[E:=\{\theta: \partial_\theta \chi(t,\theta)\nearrow\infty\},\qquad C:=\{\theta: \partial_\theta \chi(t,\theta)\rightarrow 0\}.\]
 Again by \eqref{2 pi integral}, we have that $|E|=0$ and $|C|=2\pi$. We now consider the case where $g$ does not converge weakly to a constant in $L^p$ for $p<+\infty$.  
 Then there exists a sequence of times $t_j\rightarrow\infty$ for which $\lim_{j\rightarrow \infty}\int(\partial_\theta G)^2|_{t=t_j}=c_0>0.$  We recall that by definition $E$ be the set of points $\theta\in\mathbb{S}^1$ for which $\partial_\theta G(t,\chi(t,\theta))\geq 0$ for all $t\geq 0.$ Assume now that $E$ contains $N$ distinct ordered points $\theta_1<\theta_2<\theta_3<...<\theta_N<\theta_{N+1}=\theta_1+\frac{2\pi}{m}.$
 \begin{lemma}\label{lem prf rlx thm:lower bound on E}
 Under the preceding assumptions, there exists $\delta(c_0, \norm{g_0}_{L^\infty})$
 \[\partial_\theta G(t_j,\chi(t_j,\theta_n))\geq \delta(c_0, \norm{g_0}_{L^\infty}).\]
 \end{lemma}
Before establishing the Lemma, let us first show that it implies a uniform bound on $N$. Since the sets $C_k$ increase to $C$, which is of full measure, we may take $k$ sufficiently large so that \[\abs{C_k}\geq 2\pi -\frac{1}{2}\inf_n |\theta_n-\theta_{n+1}|.\] It follows that $C_k\cap (\theta_n,\theta_{n+1})\not\neq \emptyset$ for each $n.$ Fix points $\theta_{n}'\in C_k\cap (\theta_n,\theta_{n+1}).$ Observe that, by continuity, \[\chi(t,\theta_{n}')\in (\chi(t,\theta_n), \chi(t,\theta_{n+1}))\] for all $t\geq 0.$ Now, if $t\geq k,$
\[\partial_\theta G(t,\chi(t,\theta_n')))\leq 0,\] for every $n$ since $\theta_n'\in C_k.$ Now, for each $t,$ there exists $n$ so that \[|\chi(t,\theta_n)-\chi(t,\theta_{n}')|\leq |\chi(t,\theta_n)-\chi(t,\theta_{n+1})|\leq \frac{2\pi}{N}.\]
It follows that \[|\partial_\theta G(t,\chi(t,\theta_n))-\partial_\theta G(t,\chi(\theta_{n'},t))|\leq |\partial_{\theta\theta}G|_{L^\infty}|\chi(t,\theta_n)-\chi(t,\theta_{n+1})|\leq \frac{C\norm{g}_{L^\infty}}{N},\] for all $t\geq k$. On the other hand, for $j$ such that $t_j\geq k$ we have $\partial_\theta G(\chi(\theta_{n'},t_k),t_j)\leq 0$ and $\partial_\theta G(t_j,\chi(t_j,\theta_n))\geq \delta(c_0,\norm{g_0}_{L^\infty}).$
It follows that 
\[\delta(c_0,\norm{g_0}_{L^\infty})\leq \frac{C\norm{g_0}_{L^\infty}}{N}\Longrightarrow N\leq \frac{C\norm{g_0}_{L^\infty}}{\delta(c_0,\norm{g_0}_{L^\infty})}.\]
 \begin{proof}[Proof of Lemma \ref{lem prf rlx thm:lower bound on E}]
 First note that by definition 
 \[\partial_\theta G(t_j,\chi(t_j,\theta_n))\geq 0,\]
and the goal is get a uniform lower bound. We start from
\[\frac{d}{dt} \int_{\mathbb{S}^1}(\partial_\theta G)^2 \leq C(\norm{g_0}_{L^\infty}),\]
thus if $\int_{\mathbb{S}^1}(\partial_\theta G)^2|_{t=T}\geq c_0$ 
\[\int_{\mathbb{S}^1}(\partial_\theta G)^2\geq c_0-(t-T) C(\norm{g_0}_{L^\infty})\] for all $t\geq T.$
Thus by Lemma \ref{elliptic estimate} 
\[c(t,\theta_n)\geq \tilde{c}_0-(t-t_j) \tilde{C}(\norm{g_0}_{L^\infty})\] 
for all $t\geq t_j$. Finally we recall that
\[
\frac{d}{dt}\partial_\theta G(t_j,\chi(t_j,\theta_n))=\prt{\partial_\theta G(\chi(t,\theta_n),t)}^2-c(t,\theta_n).
\]
the result then follows from the following elementary observation.
 \end{proof}

\begin{lemma}
Fix $c_0<1$ and $C_0>1$. 
Assume that $c(t,\theta)\geq c_0-C_0 (t-T)$ for $t\geq T.$ Then, if $f$ satisfies:
\[\partial_t f=f^2-c\] and if $f(T)\leq \frac{c_0^2}{1000C_0},$ then $f(t)<0$ for some $t\in [T,T+\frac{c_0}{2C_0}].$  
\end{lemma}
\begin{proof}
On the time interval $[T,T+\frac{c_0}{2C_0}]$, we have that $c\geq \frac{1}{2}c_0.$ Thus, on this time interval
\[\partial_t f\leq f^2-\frac{1}{2}c_0.\] So long as $f(t)^2\leq \frac{1}{2}c_0,$ $f$ is decreasing. This is true for some short time by continuity. While it's true,
\[\partial_t f\leq \frac{1}{4}c_0^2-\frac{1}{2}c_0\leq -\frac{1}{4}c_0.\] Thus, 
\[f(T+t)\leq f(T)-\frac{1}{4}c_0(t-T)\leq \frac{c_0^2}{1000 C_0}-\frac{1}{4}c_0(t-T).\]
Taking $t=T+\frac{c_0}{2C_0},$ we get that $f$ must become negative. 
\end{proof}

\subsubsection{Relaxation to jump states when $g_0$ admits limits on $E$}

We consider without loss of generality we work in the case where $g$ does not relax to a constant and write $E=\{a_{0}=-\frac{\pi}{m}<a_2<a_4<\cdots<a_{2n}=\frac{\pi}{m}\}$, then we have the following.
\begin{lemma}\label{lem:prf rlx jump}
For $i\in (1,n)$ there exists $a_{2i-1}\in (a_{2(i-1)},a_{2i})$ such that for all $\theta \in (a_{2(i-1)},a_{2i})$ 
\[
\abs{\chi(t,\theta)-\chi(t,a_{2i-1})}\leq \frac{C(\norm{g_0}_{L^\infty})}{y(t,\theta)^2}\underset{t\rightarrow+\infty}{\longrightarrow}0,
\]
\end{lemma}
\begin{proof}
The existence of $a_{2i-1}$ follows from the observation that $C_i=\cup_{i=1}^{2n}(a_{2(i-1)},a_{2i})$ and that $\partial_{\theta}\chi(t,\cdot)_{|(a_{2(i-1)},a_{2i})}$ converges point-wise to $0$.
\end{proof}
Fixing a test function $\psi\in C^\infty(\mathbb{S}^1)$ we compute
\begin{align*}
    &\int_{\mathbb{S}^1}g(t,\theta)\psi(\theta)d\theta=\int^{\chi^{-1}(t,\frac{\pi}{m})}_{\chi^{-1}(t,-\frac{\pi}{m})}g_0(\chi^{-1}(t,\theta))\psi(\theta)d\theta=\int^{\frac{\pi}{m}}_{-\frac{\pi}{m}}g_0(\theta)\psi(\chi(t,\theta))\partial_{\theta}\chi(t,\theta)d\theta\\
    &=\sum^{2n-1}_{i=1}\prt{\int^{a_{2i}}_{a_{2i}-\epsilon}+\int^{a_{2i}+\epsilon}_{a_{2i}}}g_0(\theta)\psi(\chi(t,\theta))\partial_{\theta}\chi(t,\theta)d\theta+O\prt{\sup_i \abs{\chi(t,a_{2i\pm1})-\chi(t,a_{2i}\pm\epsilon)}}\\
   &+\prt{\int^{a_{2n}}_{a_{2n}-\epsilon}+\int^{a_{0}+\epsilon}_{a_{0}}}g_0(\theta)\psi(\chi(t,\theta))\partial_{\theta}\chi(t,\theta)d\theta\\
     &=\sum^{2n-1}_{i=1}\prt{\prt{\lim_{\epsilon'\rightarrow 0}g_0(a_{2i}-\epsilon')}\int^{a_{2i}}_{a_{2i}-\epsilon}+\prt{\lim_{\epsilon'\rightarrow 0}g_0(a_{2i}+\epsilon')}\int^{a_{2i}+\epsilon}_{a_{2i}}}\psi(\chi(t,\theta))\partial_{\theta}\chi(t,\theta)d\theta\\
    &+\prt{\prt{\lim_{\epsilon'\rightarrow 0}g_0(a_{2n}-\epsilon')}\int^{a_{2n}}_{a_{2n}-\epsilon}+\prt{\lim_{\epsilon'\rightarrow 0}g_0(a_{0}+\epsilon')}\int^{a_{0}+\epsilon}_{a_{0}}}\psi(\chi(t,\theta))\partial_{\theta}\chi(t,\theta)d\theta\\
    &+O\prt{\delta_\epsilon+\sup_i \abs{\chi(t,a_{2i\pm1})-\chi(t,a_{2i}\pm\epsilon)}}.
\end{align*}
which gives the desired result by first passing to the limit in time then $\epsilon$ and and finally changing variables.
\subsubsection{Relaxation to the average when $E$ contains a continuity point}
We suppose by contradiction that $g$ does not converges weakly to a constant. From the previous construction if $E$ contains a continuity point $a_{2i-1}$ for $i\in \{1,\cdots,n\}$, then $g$ converges to a jump profile that is constant on $(a_{2(i-1)},a_{2i})$ with $\partial_\theta G(t,\chi(t,a_{2(i-1)}))$ and $\partial_\theta G(t,\chi(t,a_{2i}))$ are strictly positive by Lemma \ref{lem prf rlx thm:lower bound on E}. As the next Lemma \ref{lem:prf rlx cst c0} shows this implies that $\partial_\theta G(t,\chi(t,\theta))$ is always positive on $(a_{2(i-1)},a_{2i})\subset C$ thus $\partial_\theta G(t,\chi(t,\cdot))$ converges to $0$ on $(a_{2(i-1)},a_{2i})$ and by continuity on $[a_{2(i-1)},a_{2i}]$ which is a contradiction.
\begin{lemma}\label{lem:prf rlx cst c0}
Assume that $\partial_{\theta\theta}G+4G=g_i$ for some constant $g_i\in \mathbb{R}$ on an interval $(b_{i-1}=\chi(t,a_{2(i-1)}),b_i=\chi(t,a_{2i}))$ of length less than $\frac{\pi}{4}$, then if $\partial_\theta G(b_{i-1})$ and $\partial_\theta G(b_{i})$ are strictly positive then $\partial_\theta G$ is strictly positive on $(b_{i-1},b_i)$.
\end{lemma}
\begin{proof}
We write 
\[
G=\frac{g_i}{4}+\frac{A}{2}\sin(2(\theta-b_{i-1}))+\frac{B}{2}\sin(2(\theta-b_{i})),
\]
the positivity conditions then reads
\[
A>0 \text{ and } A\sin(2(b_i-b_{i-1}))+B\sin(2(b_i-b_{i}))>0
\]
since $b_i-b_{i-1}<\frac{\pi}{4}$ this last condition is equivalent to 
\[
A\cot(2(b_i-b_{i-1}))+B>0.
\]
The result then follows by observing that $\cot$ is strictly increasing on $(0,\frac{\pi}{4})$ which gives for $\theta\in (b_{i-1},b_i)$
\[
A\cot(2(\theta-b_{i-1}))+B>0.
\]
\end{proof}
\subsection{Proof of Theorem \ref{Steady states of the zero homogeneous Euler equation}}
\begin{proof}
For this we study the steady equation
\[
G\partial_\theta g=0 \text{ with } 4G+\partial_{\theta \theta}G=g.
\]
\paragraph{$\bullet$ The Riccati structure and detailed analysis of the expanding and contracting sets.} 
The equation on $\partial_\theta G$ is again given by
\[
  2G\partial_\theta \partial_\theta G-(\partial_\theta G)^2+3\dashint (\partial_\theta G)^2
+\frac{12}{\partial_{\theta \theta}+4} \overline{(\partial_\theta G)^2}=0.
\]
Now define the steady flow map
\[
\frac{d}{dt}\chi=2G\circ \chi, \ \  \chi(0,\cdot)=Id.
\]
Of course for steady states we have
\[
g(\chi(t,\theta))=g(\theta).
\]
Now the exact same analysis as in the non steady case gives by the uniqueness 0f the weak limit that either $g$ is constant or the expanding set is finite and $g$ is piecewise constant with a finite number of jumps. 
\paragraph{$\bullet$ Non trivial Steady states.}
In this case the previous analysis gives that the contracting set $C=\cup_{1\leq i\leq 2n}C_i$ where $C_i$ are the open connected components of $C$ and $\abs{C}=2\pi$. Moreover $g$ is constant on $C_i$, and we denote $g_i$ this constant, multiplying 
\[
4G+\partial_{\theta \theta}G=g,
\] 
by $\partial_\theta G$ and integrating between the lower boundary $\underline{C_i}$ and upper boundary $\overline{C_i}$ of $C_i$ we get
\[
\partial_\theta G(\underline{C_i})^2=\partial_\theta G(\overline{C_i})^2.
\]
Moreover $\partial_{\theta \theta}G$ being signed implies $\sgn \prt{\partial_\theta G(\underline{C_i})}=-\sgn \prt{\partial_\theta G(\overline{C_i})}$, thus the jumps occur at the global minima and maxima of $\partial_\theta G$ and they have exactly opposite values. Observe also that it follows that 
\begin{equation}\label{steady state balance}
\text{if } \exists i, \ g_i=0, \text{ then }\partial_\theta G(a_i)=0 \Longrightarrow \forall j, \ g_j=0. 
\end{equation}
Thus for non trivial steady states $g$ necessarily jumps from non a zero value to another non zero value.
\paragraph{$\bullet$ Steady states with a finite number of jumps: a system of ODEs}
We consider a sequence of points $a_0=-\frac{\pi}{m}< a_1<\cdots<a_{n-1}<a_n=\frac{\pi}{m}$ and a sequence of real numbers $(g_i)_{1\leq i \leq {2n}}\in \mathbb{R}$ that defines for $\theta \in (-\frac{\pi}{m},\frac{\pi}{m})$ a piecewise constant profile given by 
\[
g_0(\theta)=\sum^{2n}_{i=1}g_i\mathbbm{1}_{(a_{i-1},a_i)}(\theta).
\]
We extend $g_0$ to $[-\pi,\pi]$ by m-fold symmetry.
The solution to the system 
\[
\begin{cases}
\partial_t g+2G\partial_\theta g=0\\
4G+\partial_{\theta \theta}G=g
\end{cases}, \text{ with } g(0,\cdot)=g_0, \text{ is given by }
\]
\[
g(t,\theta)=\sum^{2n}_{i=1}g_i\mathbbm{1}_{(a_{i-1}(t),a_i(t))}(\theta)
\]
where $a_i(t)$ verify a system of ODEs that we will write explicitly. First we compute
\begin{align*}
G(t,\theta)&=\frac{C_m m}{2\pi}\int_{-\frac{\pi}{m}}^{\frac{\pi}{m}}\left|\sin\left(\frac{m}{2}(\theta-\theta')\right)\right|g(\theta')d\theta'+\frac{\tilde{C}_m m}{2\pi}\int_{-\frac{\pi}{m}}^{\frac{\pi}{m}}g(\theta')d\theta'\\
&=\frac{C_m m}{2\pi}\sum^{2n}_{i=1}g_i\int_{a_{i-1}}^{a_i}\sin\left(\frac{m}{2}(\theta-\theta')\right)\prt{\mathbbm{1}_{\theta\geq \theta'}-\mathbbm{1}_{\theta\leq \theta'}}d\theta'+\frac{\tilde{C}_m m}{2\pi}\sum^{2n}_{i=1}g_i(a_i-a_{i-1})
\end{align*}
We will only need to know $G$ on $a_j(t)$ for $j\in \{1,\cdots,{2n}\}$ thus
\begin{equation*}
G(t,a_j(t))=\frac{C_m }{\pi}\sum^{2n}_{i=1}g_i\abs{\cos\left(\frac{m}{2}(a_{i}(t)-a_j(t))\right)-\cos\left(\frac{m}{2}(a_{i-1}(t)-a_j(t))\right)}+\frac{\tilde{C}_m m}{2\pi}\sum^{2n}_{i=1}g_i(a_i(t)-a_{i-1}(t))
\end{equation*}
We also compute.
\[
\partial_\theta g(t,\theta)=-\sum^{2n}_{i=1}(g_{i+1}-g_i)\delta_{a_{i}(t)}(\theta),
\]
\[
\partial_t g(t,\theta)=\sum^{2n}_{i=1}(g_{i+1}-g_i)a'_i(t)\delta_{a_{i}(t)}(\theta).
\]
Thus we get the system for $i\in \{1,\cdots,{2n}\}$
\[
(g_{i+1}-g_i)a'_i(t)=2(g_{i+1}-g_i)G(t,a_{i}(t)),\ a_{i}(0)=a_{i},
\]
which for $j\in \{1,\cdots,{2n}\}$ reads, with initially $a_j(0)=a_j$,
\begin{multline*}
a'_j(t)=\frac{2C_m }{\pi}\sum^{2n}_{i=1}g_i\abs{\cos\left(\frac{m}{2}(a_{i}(t)-a_j(t))\right)-\cos\left(\frac{m}{2}(a_{i-1}(t)-a_j(t))\right)}+\frac{\tilde{C}_m m}{\pi}\sum^{2n}_{i=1}g_i(a_i(t)-a_{i-1}(t))
\end{multline*}
Steady states then necessarily verify, $a_0=-\frac{\pi}{m}<a_1<\cdots<a_{n-1}<a_n=\frac{\pi}{m}$ and for $j\in \{1,\cdots,2n\}$
\begin{equation*}
\frac{2C_m }{\pi}\sum^{2n}_{i=1}g_i\abs{\cos\left(\frac{m}{2}(a_{i}(t)-a_j(t))\right)-\cos\left(\frac{m}{2}(a_{i-1}(t)-a_j(t))\right)}+\frac{\tilde{C}_m m}{\pi}\sum^{2n}_{i=1}g_i(a_i(t)-a_{i-1}(t))=0
\end{equation*}
\paragraph{$\bullet$ Uniqueness of steady states}\begin{lemma}
System \eqref{System for Steady states} admits at most one steady state solution. Moreover if the $(g_i)$ do not \underline{strictly} alternate sign then the only possible steady state is $0$. 
\end{lemma}
\begin{proof}
Computing $G$ on $(a_{i-1},a_i)$ and $(a_{i},a_{i+1})$ we get 
\[
G(\theta)=g_{i}\prt{1+\frac{\sin\prt{2(\theta-a_i)}-\sin\prt{2(\theta-a_{i-1})}}{\sin\prt{2(a_i-a_{i-1})}}} \text{ on } (a_{i-1},a_i),
\]
\[
G(\theta)=g_{i+1}\prt{1+\frac{\sin\prt{2(\theta-a_{i+1})}-\sin\prt{2(\theta-a_{i})}}{\sin\prt{2(a_{i+1}-a_{i})}}} \text{ on } (a_{i},a_{i+1}).
\]
Imposing $G$ being $C^1$ at $a_i$ we get 
\[
g_i\tan(a_i-a_{i-1})=-g_{i+1}\tan(a_{i+1}-a_{i}),
\]
thus outside the $0$ state $g_i$ have to strictly change signs. Moreover we note that all of the $a_i$ are completely determined by $a_1-a_0$. Thus the identity $a_{2n}-a_0=\frac{2\pi}{m}$ impose uniqueness.
\end{proof}
\end{proof}

\section{Application to spiral formation on the Sphere}\label{sec:appl sphere}
\subsection{The equations}
We start by writing the Euler equations on the sphere in spherical coordinates. First in velocity form we have
\begin{equation*}
\partial_t u+(u\cdot \nabla_{\mathbb{S}^2})u=-\nabla_{\mathbb{S}^2} p \text{ with } \div(u)=0.
\end{equation*}
Defining the vorticity  such that $d_{\mathbb{S}^2}(u,\cdot)_{\mathbb{S}^2}=\omega d_{\mathbb{S}^2}V$ where $d_{\mathbb{S}^2}$ and $(\cdot,\cdot)_{\mathbb{S}^2}$ are the exterior derivative and the first fundamental form on $\mathbb{S}^2$ . Then 
\begin{equation*}
\partial_t \omega+u\cdot \nabla_{\mathbb{S}^2} \omega =0
\text{ and }
u=\nabla_{\mathbb{S}^2} e_r \times \psi \ \text{ with } \ \Delta_{\mathbb{S}^2} \psi =\omega,
\end{equation*}
where $\psi$ is the stream function.
Choosing the convention with $r>0$, $\theta \in (0,\pi)$, $\varphi \in [0,2\pi)$ and


\[
r=\sqrt{x^2+y^2+z^2}, \ \theta=\arccos\prt{\frac{z}{r}}, \ \varphi=\arctan_2\prt{\frac{y}{x}}.
\]
Then the resulting basis of vectors reads
\[
\begin{cases}
e_r=\frac{x}{r}e_x+\frac{y}{r}e_y+\frac{z}{r}e_z, \\
e_\theta=\frac{xz}{r}\frac{1}{\sqrt{x^2+y^2}}e_x+\frac{yz}{r}\frac{1}{\sqrt{x^2+y^2}}e_y-\frac{\sqrt{x^2+y^2}}{r}e_z, \\
e_\varphi=-\frac{y}{\sqrt{x^2+y^2}}e_x+\frac{x}{\sqrt{x^2+y^2}}e_y.
\end{cases}\]

\begin{lemma}\label{Sphere:spherical coord}
In the previous convention we have 
\[
e_r\times e_\theta=e_\phi, \ \ e_{\theta}\times e_\varphi=e_r , \ \ e_{\varphi}\times e_r=e_\theta, \\ \nabla_{\mathbb{S}^2}\cdot =\begin{pmatrix}
0 \\
\ \ \partial_\theta\cdot \\
\frac{1}{\sin(\theta)}\partial_\varphi\cdot 
\end{pmatrix}_{(e_r,e_\theta,e_\varphi)},
\]
\[
e_r \times \nabla_{\mathbb{S}^2} \cdot =\nabla^\perp_{\mathbb{S}^2}\cdot =\begin{pmatrix}
0 \\
\ \ -\frac{1}{\sin(\theta)}\partial_\varphi\\
\partial_\theta
\end{pmatrix}_{(e_r,e_\theta,e_\varphi)} \text{ and } \Delta_{\mathbb{S}^2}\cdot  =\frac{1}{\sin(\theta)}\partial_\theta\prt{\sin(\theta)\partial_\theta \cdot  }+\frac{1}{\sin(\theta)^2}\partial_\varphi^2 \cdot 
\]
\end{lemma}

Thus for $\theta \in (0,\pi)$ and $\varphi \in [0,2\pi)$ the 2d Euler equations on the sphere read

\begin{equation}\label{eq:2d E on the sphere}
\begin{cases}
    \partial_t \omega+u_\theta \partial_\theta \omega+\frac{u_\varphi}{\sin(\theta)}\partial_\varphi \omega=0\\
    \frac{1}{\sin(\theta)}\partial_\theta\prt{\sin(\theta)\partial_\theta \psi}+\frac{1}{\sin(\theta)^2}\partial_\varphi^2 \psi=\omega\\
    u_\theta=-\frac{1}{\sin(\theta)}\partial_\varphi \psi, \ \   u_\varphi=\partial_\theta \psi
    \end{cases}.
\end{equation}
\subsection{The rotation around the poles in $m$-fold symmetry with $m\geq 3$}
Keeping the conventions of the previous subsection we will denote by the subscript $m$ the space of functions on $\mathbb{S}^2$ that are $m-$fold symmetric with respect to the $z$ axis. This a discrete symmetry propagated by the Euler flow.
\begin{proposition}\label{prop: movement of poles}
Consider $\omega_0\in L^\infty_m(\mathbb{S}^2)$ and the associated unique global solution $\omega \in C_w\prt{\Rr,L_m^{\infty}(\mathbb{S}^2)}$ of \eqref{eq:2d E on the sphere}. Suppose that there exists $(g_0^N, g_0^S)\in L^\infty_m(\mathbb{S}^1)^2$ such that for almost every $\varphi\in[0,2\pi)$
\[
\lim_{\theta\rightarrow 0}\omega_0(\theta,\varphi)=g_0^N(\varphi) \text{ and } \lim_{\theta\rightarrow \pi}\omega_0(\theta,\varphi)=g_0^S(\varphi),
\]
while $g_0^N, g_0^S\in W^{1,\infty}(\mathbb{S}^1).$
Then there exist $g^N, g^S\in C_w\prt{\Rr,W^{1,\infty}_m(\mathbb{S}^1)}$ such that for all $t\in \Rr$ and almost every $\varphi\in[0,2\pi)$
\[
\lim_{\theta\rightarrow 0}\omega(t,\theta,\varphi)=g^N(t,\varphi) \text{ and } \lim_{\theta\rightarrow \pi}\omega(t,\theta,\varphi)=g^S(t,\varphi).
\]
Moreover $g^N$ solves \eqref{0 homogeneous limit system with initial data} with initial data $g^N_0$ forward in time and $g^S$ solves \eqref{0 homogeneous limit system with initial data} with data $g^S_0$ backward in time. 
\end{proposition}

\begin{proof}
We start with the key Lemma. 
\begin{lemma}\label{lem:boundness of BS}
Consider $\omega\in L^\infty_m(\mathbb{S}^2)$ then $u=\nabla^\perp_{\mathbb{S}^2}\prt{\Delta_{\mathbb{S}^2}}^{-1}\omega$ is point-wise well defined, Log Lipschitz and 
\[
\abs{u(\theta,\varphi)}\leq C \abs{\sin(\theta)} \norm{\omega}_{L^\infty}.
\]
Moreover we have 
\begin{equation}\label{eq:cancellation at pole in symmetry}
\lim_{\theta\rightarrow 0}\omega(\theta,\varphi)=0\Longrightarrow \lim_{\theta\rightarrow 0} \frac{\abs{u(\theta,\varphi)}}{\sin(\theta)}=0,
\end{equation}
and the analogously at $\theta=\pi$.
\end{lemma}
\begin{proof}[Proof of Lemma \ref{lem:boundness of BS}]
The fact that $u$ is Log Lipschitz follows from the classical theory of Yudovich on bounded manifolds (see, for example, the work of Taylor \cite{Taylor}). Now considering the Biot-Savart Law in stereographic projection, that is we make the change of variable $r=\cot\prt{\frac{\theta}{2}}$, we get 
\[
\partial_\theta =-\frac{1}{2}\frac{1}{\sin\prt{\frac{\theta}{2}}^2}\partial_r=-\frac{1+r^2}{2}\partial_r \text{ and } \sin(\theta)=\frac{2r}{1+r^2}.
\]
The Biot-Savart law then becomes 
\[\underbrace{\frac{1}{r}\partial_r\prt{r\partial_r \psi}+\frac{1}{r^2}\partial_\varphi^2 \psi}_{:=\Delta_{\mathbb{R}^2}\psi}=\frac{4\omega}{(1+r^2)^2}\]
Now $\Delta_{\mathbb{R}^2}$ is exactly the Laplacian on $\mathbb{R}^2$ in polar coordinates in the variable $(r,\varphi)$ and the point-wise well definition and the estimate follow in verbatim from Lemma 2.10 in \cite{EJ-symmetries}. To get the more refined estimate \eqref{eq:cancellation at pole in symmetry} we note that using the stereo-graphic projection it suffice to prove the result on $\mathbb{R}^2$ with a bounded $m-$fold symmetric vorticity vanishing at the origin which plays the role of the north pole.  We recall from \cite{EJ-symmetries} that if $K$ denotes the Biot-Savart Kernel on $\mathbb{R}^2$
\[
K(x,y):=\frac{1}{2\pi}\frac{(x-y)^\perp}{\abs{x-y}^2}
\]
then for $m\geq 1$, the m-fold symmetrization in $y$ is defined by
\[
K_m(x,y)=\frac{1}{m}\sum^m_{i=1}K(x,O^i_{\frac{2\pi}{m}}y).
\]
Then by Corollary 2.15 in \cite{EJ-symmetries} for $\abs{y}\geq 2\abs{x}$ we have 
\[
K_m(x,y)\leq C_m \frac{\abs{x}^{m-1}}{\abs{y}^m}.
\]
We now write for $\lambda \in \mathbb{R}_+$
\[
u(x)=\int_{\mathbb{R}^2}K_m(x,y)\omega(y)dy=\underbrace{\int_{\abs{y}\leq \lambda \abs{x}}K_m(x,y)\omega(y)dy}_{(1)}+\underbrace{\int_{\abs{y}\geq \lambda \abs{x}}K_m(x,y)\omega(y)dy}_{(2)}.
\]
For $\epsilon \geq 0$ there exists $\delta$ such that for $\lambda \abs{x}\leq \delta$
\[
\abs{(1)}\leq \frac{1}{2\pi}\int_{\abs{y}\leq \lambda \abs{x}}\frac{1}{\abs{x-y}}w(y)dy\leq C \lambda \abs{x} \epsilon \text{ and }
\]
\[
\abs{(2)}\leq C_m \abs{x}^{m-1}\int_{\abs{y}\geq \lambda \abs{x}}\frac{1}{\abs{y}^m}dy\leq C_m\lambda^{2-m}\abs{x}.
\]
Thus for $\lambda \abs{x}\leq \delta$
\[
\frac{\abs{u(x)}}{\abs{x}}\leq C \lambda \epsilon+ C_m\lambda^{2-m},
\]
which gives the desired result by taking $\epsilon$ to $0$ and $\lambda$ to infinity.
\end{proof}
Now, with the Lemma in hand, the result is relatively straightforward. Indeed, we write 
\[\psi(\theta,\phi)=\sin^2(\theta)\Psi(\theta,\phi)\] and the equations transform from 
\[\partial_t \omega +\frac{1}{\sin(\theta)}(\partial_\theta\psi\partial_\phi\omega-\partial_\phi\psi\partial_\theta\omega)=0\]
to
\[\partial_t \omega +\frac{1}{\sin(\theta)}(\partial_\theta(\sin^2(\theta)\Psi)\partial_\phi\omega-\sin^2(\theta)\partial_\phi\Psi\partial_\theta\omega)=0\]
\[
\frac{1}{\sin(\theta)}\partial_\theta( \sin(\theta)\partial_\theta (\sin^2(\theta)\Psi))+\partial_{\phi}^2\Psi=\omega,\]
which simplify to 
\[\partial_t\omega +2\cos(\theta)\Psi\partial_\phi\omega+\sin(\theta)(\partial_\theta\Psi\partial_\phi\omega-\partial_\phi\Psi\partial_\theta\omega)=0,\]
\[ \sin^2(\theta)\partial_{\theta\theta}\Psi+3\sin(2\theta)\partial_\theta\Psi+4\cos(2\theta)\Psi +\partial_\phi^2\Psi =\omega.\] Formally taking the limit $\theta\rightarrow 0$ gives 
\[\partial_tg^N+2\Psi\partial_\phi g^N=0,\]
\[4\Psi^N+\partial_{\phi}^2\Psi^N=g^N,\] 
and formally taking the limit $\theta\rightarrow \pi$ gives
\[\partial_tg^S-2\Psi\partial_\phi g^S=0,\]
\[4\Psi^S+\partial_{\phi}^2\Psi^S=g^S.\] The justification of these limits (as in \cite{EJ-symmetries}) rests simply on writing the equation for  $\tilde\omega^N:=\omega-g^N$ and observing that, using Lemma \ref{lem:boundness of BS} satisfies
\[\lim_{\theta\rightarrow 0} \tilde\omega^N=0\] and similarly at the south pole. We leave the details to the interested reader.

\end{proof}

\section{Proof of Theorems \ref{R2Theorem} and \ref{S2Theorem}}
The proofs of these two theorems are exactly the same and rely upon Lemmas \ref{Koch} and \ref{LagrangianChaos} given below (both needing to be simply modified to establish Theorem \ref{S2Theorem}) and an application of the Baire Category Theorem. Thus we only give the full proof of Theorem \ref{R2Theorem}.

We first recall a result of Koch
\begin{lemma}[\cite{Koch}]\label{Koch}
Assume that $\omega\in L^1\cap C^2_m$ is a solution to the 2d Euler equation with Lagrangian flow $\Phi_\omega.$ For $T\geq 0,$ let $\mu(T)=\sup_{x}|\nabla \Phi_\omega(x,T)|.$ Then, for any $\epsilon>0$ and $T\geq 0,$ there exists a solution $v\in L^1\cap C^1_m$ with 
\[\norm{\omega_0-v_0}_{L^1\cap C^1}\leq \epsilon,\] while 
\[\norm{\omega-v}_{C^1}\geq \frac{\epsilon}{2} \mu(T).\]
\end{lemma}

\begin{lemma}\label{LagrangianChaos}
Assume that $\omega_0\in L^1\cap C^1_{m}(\mathbb{R}^2)$ and that $\omega_0(0)\not=0.$ Then, \[\sup_{t}\norm{\nabla\Phi}_{L^\infty}=+\infty,\] while $|\nabla\Phi(0,t)|=1$ for all time. Moreover, the image of the $x$-axis under the flow $\Phi(\cdot,t)$ intersects any line passing through the origin at least $c_0\cdot t$ times as $t\rightarrow\infty,$ for some $c_0=c_0(\omega_0).$ 
\end{lemma}

\begin{proof}[Proof of Lemma \ref{LagrangianChaos}]
The proof of spiral formation in this case is actually contained in the proof of Theorem 2 in \cite{EJSVPII}, so we will only give the argument for the infinite growth of $\nabla \Phi.$
Assume without a loss of generality what $\omega_0(0)=1$ and that $|u(x,t)|_{L^\infty}\leq 1.$ Then, $\omega(0,t)=1$ for all $t\in\mathbb{R}.$ By the $m$-fold symmetry assumption and the assumption, we have that $\sup_t|\omega|_{C^1}<\infty$. We thus have that
\begin{equation}\label{velocityexpansion} u(x,t)=x^\perp + h(x,t),\end{equation} where
\[|h(x,t)|\leq M|x|^2,\] for a fixed constant $M>0$ and for all $(x,t)\in\mathbb{R}^2\times\mathbb{R}.$ Assume now, toward a contradiction, that \[|\nabla\Phi(x,t)|\leq K\] for all $(x,t)\in \mathbb{R}^2\times \mathbb{R}.$ Since $\det{\nabla\Phi}=1,$ we also have that $\norm{\nabla\Phi^{-1}}_{L^\infty}\leq K.$ Since $\Phi(0,t)=0$ for all time, it follows that  
\[\frac{1}{K}\leq \frac{|\Phi(x,t)|}{|x|}\leq K.\] It follows that \[|x|\leq \frac{1}{100 M K}\implies |\Phi(x,t)|\leq \frac{1}{100M},\] while
\[|x|\geq 100K\implies |\Phi(x,t)|\geq 100,\]for all $t\geq 0.$ Now, for each $x$ Define $\theta(x,t)$ to be the (unique continuous) angle function that $\Phi(x,t)$ makes with the $x_1$-axis.  Then, we have from \eqref{velocityexpansion} that if $|x|\leq \frac{1}{100MK},$ we have that 
\[|\frac{d}{dt}\theta(x,t)-1|\leq \frac{1}{100}.\] It thus follows that 
\begin{equation}\label{innerevolution}|x|\leq \frac{1}{100MK}\implies -2\pi+0.99t\leq \theta(x,t)\leq 1.01 t+2\pi,\end{equation}
and, similarly, because $|u|_{L^\infty}\leq 1,$
\begin{equation}\label{outerevolution}|x|\geq 100K\implies \theta(x,t) \leq 0.01 t+2\pi.\end{equation} 
Now consider the quarter of an annulus $A_0$ defined \emph{in polar coordinates} $(\theta,r)$ by:
\[A(0)= [0,\pi/2]\times [\frac{1}{100MK}, 100K] .\] Let $A(t)=\Phi(A(0),t).$ Note that the angle is defined on $\mathbb{R}$ (as the universal cover of $\mathbb{S}^1$) and that by \eqref{innerevolution} and \eqref{outerevolution}, we have that the evolution of the curve $[0,\pi/2]\times \{\frac{1}{100MK}\}$ is contained in $\theta\geq -2\pi+0.99t$ while the curve $[0,\pi/2]\times \{100K\}$ is contained in $\theta\leq 2\pi+0.01t.$ By Fubini's theorem, it follows that the distance between the image of the line $\Gamma_0:=\{0\}\times [\frac{1}{100MK}, 100K]$ and $\Gamma_{\pi/2}:=\{\pi/2\}\times [\frac{1}{100MK}, 100K]$ satisfies
\[\text{dist}(\Phi(\Gamma_0,t), \Phi(\Gamma_{\pi/2},t))\leq \frac{c}{t},\] for some $c=C(M,K).$ This implies $\norm{\nabla \Phi}_{L^\infty}\geq c t$ as $t\rightarrow\infty,$ which is a contradiction. 

\end{proof}
\begin{proof}[Proof of Theorem \ref{R2Theorem}]
For this argument, we are taking inspiration from the work of Z. Hani \cite{Hani} on the nonlinear Schr\"odinger equation. 
Let $S_t$ denote the solution map that sends an initial data to the corresponding Euler solution at time $t$. It is easy to show that $S_t:L^1\cap C^1_m\rightarrow L^1\cap C^1_m$ is continuous for all $t\in\mathbb{R}.$ Now, for each $N\in\mathbb{N},$ we define:
\[B_N:=\{\omega_0\in L^1\cap C^1_m(\mathbb{R}^2): \sup_{t\geq 0}\norm{S_t(\omega_0)}_{C^1}>N\}.\] By continuity of the solution map $S_t$ in $C^1,$ $B_N$ is an open subset of $L^1\cap C^{1}_m.$ We will now show that it is dense. Take some $\omega_0\in L^1\cap C^1_m$ and let $\epsilon>0.$ Then, there exists $\bar{\omega_0}\in B_{\epsilon/2}(\omega_0)$ belonging to $L^1\cap C^2_m$ and with $\bar\omega_0(0)\not=0.$ It follows from Lemma \ref{LagrangianChaos} that if $\Phi$ is the flow-map associated to $S_t(\omega_0),$
\[\sup_{t\geq 0}\norm{\nabla\Phi}_{L^\infty}=+\infty.\] By Koch's theorem, Lemma \ref{Koch}, it follows that there exists $v_0\in B_{\epsilon/2}$ and a $t\in [0,\infty)$ so that 
\[\norm{S_t(\bar\omega_0)-S_t(v_0)}_{C^1}\geq 4N.\] It follows that either $\bar{\omega_0}\in B_N\cap B_{\epsilon}(\omega_0)$ or $v_0\in B_N\cap B_\epsilon(\omega_0).$ It follows that $B_N$ is open and dense in $X$ for each $N\in\mathbb{N}$. The theorem now follows from the Baire Category Theorem. 
\end{proof}

\section*{Acknowledgements}
T.M. Elgindi and A.R. Said acknowledge funding from the NSF grants DMS-2043024 and
DMS-2124748. T.M. Elgindi also acknowledges funding from an Alfred P. Sloan fellowship. The authors thank T. Drivas for very helpful comments on a first draft of this work.

\appendix
\section{Some ODE results}
The following 3 lemmas come from \cite{Sobol49} where they were announced but not proven. As we could not get a hold of the subsequent article by the same author in which the proofs are said to have been published we give a proof to each lemma.
Consider the ordinary differential equation \begin{equation}
\label{YEqn} y''(t)=c(t)y(t),
\end{equation} where $c(t)>0$ is uniformly bounded.

\begin{lemma}\label{ODELemma1}
If $y$ is a positive solution to \eqref{YEqn}, then $\lim_{t\rightarrow\infty} y(t)$ exists. If the limit is finite, then $y(t)$ decreases to $y_\infty\in [0,\infty)$. 
\end{lemma}
\begin{proof}
Since $y''\geq 0$, it is either that $y'(t)<0$ for all $t$ or there exists $t_*$ after which $y'(t)>0$. In the latter case, $y'$ must be bounded from below so that $y\rightarrow \infty$ at least linearly.
\end{proof}

\begin{lemma}\label{ODELemma2}
Assume that $y$ solves \eqref{YEqn} and that $y$ decreases to $y_\infty\in [0,\infty)$. Then, 
\[\int_0^\infty |y'|=-\int_0^\infty y'=y_\infty-y(0).\] Furthermore, since $c(t)$ is uniformly bounded, we have that $y''$ is uniformly bounded. This implies that $y'(t)\rightarrow 0$. 
\end{lemma}

\begin{proof}
The proof is in the Lemma.
\end{proof}

\begin{lemma}\label{ODELemma3}
If $y_\infty>0$, then $\int_0^\infty t c(t) dt <\infty$. 
\end{lemma}

\begin{proof}

Write $y=e^z$. Then, 
\[z''=c(t)-z'(t)^2.\] Note that $z'e^z=y'$. Since $y_\infty>0$, we have that $z_\infty$ exists and $z'\rightarrow 0$. 
Observe that under our assumptions, we have that $z$ decreases to $z_\infty$ and that $z'<0$. Now let $z'=v$.  By the above, we have that $v\rightarrow 0$ as $t\rightarrow\infty$ and $v$ is integrable.  Then, \[v'=c(t)-v^2.\]
We will now show that there exists $M\geq 0$ so that \[|v(t)|\leq \frac{M}{t},\] for all $t$. Since $\int_0^\infty |v|<\infty,$ there exists $T_*\geq 1$ so that 
\[\int_{T_*}^\infty|v|<\frac{1}{1000}.\]
Now split $[T_*,\infty)$ into $I_k=[T_*k, T_*(k+1)]$ for $k\geq 1$. 
By Chebyshev's inequality, there exists $t_k\in I_k$ so that \[|v(t_k)|\leq \frac{1}{500t_k}.\] Next, observe that \[v'(t)\geq -v(t)^2.\] Thus, 
\[\frac{d}{dt}\Big(\frac{1}{|v|}\Big)\geq -1,\] where we just used $-v=|v|$. Thus, 
\[\frac{1}{|v(t)|}-\frac{1}{|v(s)|}\geq -(t-s).\] Therefore, 
\[|v(t)|\leq \frac{1}{\frac{1}{|v(s)|}-(t-s)}.\] Now set $s=t_k$ and take $t\in [t_k, t_k+2T_*].$ Then, we must have that
\[|v(t)|\leq \frac{1}{500t_k-2T_*}\leq \frac{1}{100t}.\] It follows that for all $t\geq 2T_*$ we have that \[|v(t)|\leq\frac{1}{100t}.\] In fact, it follows that $v=o(\frac{1}{t})$, but we do not need this. Now we simply observe that 
\[v't=tc(t)-tv^2.\] Now integrate both sides on $[0,T]$. 
Then we see that 
\[ \int_0^T tc(t)dt=\int_0^T v'(t)t +\int_0^T tv^2= v(T)T+\int_0^T v(t)dt +\int_0^Ttv^2\leq M\Big(1+2\int_0^\infty |v|\Big).\] This concludes the proof. 
\end{proof}

\end{document}